\documentclass[
final
]{dmtcs-episciences}

\usepackage[utf8]{inputenc}
\usepackage{subfigure}
\usepackage{hyperref}
\usepackage{amsfonts,amssymb,amsmath,amsthm,epsfig,euscript}
\usepackage{graphicx}

\newcommand{\inv}[1]{\mathrm{inv}(#1)}
\newcommand{\coinv}[1]{\mathrm{coinv}(#1)}

\newtheorem{theorem}{Theorem}

\newtheorem{corollary}[theorem]{Corollary}
\newtheorem{definition}[theorem]{Definition}
\newtheorem{conjecture}[theorem]{Conjecture}

\author{Ran Pan\affiliationmark{1}
	\and Jeffrey B. Remmel\affiliationmark{2}}
\title[Asymptotics for minimal overlapping patterns]{Asymptotics for minimal overlapping patterns for generalized 
	Euler permutations, standard tableaux of 
	rectangular shapes, and column strict arrays}
\affiliation{
	Department of Mathematics, UCSD
	La Jolla, CA, 92093-0112, r1pan@ucsd.edu \\
	Department of Mathematics, UCSD
	La Jolla, CA, 92093-0112, jremmel@ucsd.edu }
\keywords{permutations, arrays, minimal overlapping patterns}
\received{2015-10-29}
\revised{2016-02-26}
\accepted{2016-3-24}

\begin{document}
	
\publicationdetails{18}{2016}{2}{6}{1315}

\maketitle

\begin{abstract} A permutation $\tau$ in the symmetric group $S_j$ 
is minimally overlapping if any two consecutive occurrences of 
$\tau$ in a permutation $\sigma$ can share at most one element. B\'ona 
showed that the proportion of minimal overlapping patterns 
in $S_j$ is at least $3 -e$.  Given a permutation 
$\sigma$, we let $\text{Des}(\sigma)$ denote the set of descents of $\sigma$. We study 
the class of 
permutations $\sigma \in S_{kn}$ whose descent set is contained in the 
set $\{k,2k, \ldots (n-1)k\}$. For example, 
up-down permutations in $S_{2n}$ are the set of permutations 
whose descent equal $\sigma$ such that $\text{Des}(\sigma) = \{2,4, \ldots, 2n-2\}$. 
There are natural analogues of the minimal overlapping permutations 
for such classes of permutations and we study the proportion of 
minimal overlapping patterns for each such class.  We show that the
proportion of minimal overlapping permutations in such classes 
approaches $1$ as $k$ goes to infinity.    We also study 
the proportion of minimal overlapping patterns in standard Young tableaux 
of shape $(n^k)$. 
\end{abstract}

\section{Introduction}
Let $S_n$ denote the set of permutations of $[n]=\{1, \ldots, n\}$. 
We let $k[n] = \{k,2k, \ldots, nk\}$ if $k,n \geq 1$. If 
$\sigma = \sigma_1 \ldots \sigma_n$ is an element of $S_n$, we 
let $\text{Des}(\sigma)=\{i:\sigma_i > \sigma_{i+1}\}$. For any $k,n \geq 1$,  
$\mathcal{C}_{= k[n-1]}$ denote the set of all $\sigma$ in $S_{kn}$ such that $\text{Des}(\sigma) = k[n-1]$ and $\mathcal{C}_{\subseteq k[n-1]}$ denote the set of all $\sigma$ in $S_{kn}$ such that $\text{Des}(\sigma) \subseteq k[n-1]$. 
For example, elements of $\mathcal{C}_{= 2[n-1]}$ are permutations 
$\sigma = \sigma_1 \sigma_2 \ldots \sigma_{2n} \in S_{2n}$ such that 
$$\sigma_1< \sigma_2> \sigma_3 < \sigma_4> \cdots .$$
These are called up-down permutations or even alternating permutations.

One way to think of elements in $\mathcal{C}_{\subseteq k[n-1]}$ is 
as column strict arrays which were studied by  
Harmse and Remmel in \cite{H} and \cite{HR}. 
A column-strict array $P$ is a filling of a $k\times n$ rectangular array with $1,2,\cdots,kn$ such that elements increase from bottom to top in each column. 
Let $\mathcal{F}_{n,k}$ denote the set of all the column-strict arrays with $n$ columns and $k$ rows. Given a $P \in \mathcal{F}_{n,k}$, 
we let $P[i,j]$ denote element in the $i$th column and $j$th row 
where the columns are labeled from left to right and the rows are labeled 
from bottom to top. We let $\sigma_P$ denote 
the permutation 
$$\sigma_p =P[1,1]P[1,2]\ldots P[1,k]P[2,1]P[2,2]\ldots P[2,k] \ldots 
P[n,1]P[n,2]\ldots P[n,k].$$ 
For example, if $P$ is the element of $\mathcal{F}_{4,3}$ pictured 
in Figure \ref{fig:43array}, then 
$\sigma_P = 3~11~12~1~2~4~6~8~10~5~7~9$. This given, 
it is easy to see that $\sigma_P \in  \mathcal{C}_{=k[n-1]}$ if 
and only if $P[i,k]>P[i+1,1]$ for all $1 \leq i \leq n-1$.  
An array $P \in \mathcal{F}_{n,k}$ is a {\em standard tableau} of 
shape $(n^k)$ if $P$ is strictly increasing in rows, reading 
from left to right. That is, $P$ is strictly increasing in rows 
if for all $1 \leq j \leq k$, 
$$P[1,j]<P[2,j]< \cdots < P[n,j].$$
We let $\mathcal{ST}(n^k)$ denote the set of all standard tableaux of shape $(n^k)$. 

%\fig{1.1}{}{An element of $\mathcal{F}_{4,3}.$} 
\begin{figure}[htbp]
	\begin{center}
		%% Don't use epsfile!
		\includegraphics[width = 60pt, height = 45pt]{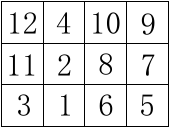}
		\caption{An element of $\mathcal{F}_{4,3}.$}
		\label{fig:43array}
	\end{center}
\end{figure}

Given any sequence of pairwise distinct positive integers, 
$\alpha = \alpha_1 \ldots \alpha_n$, we let $\text{red}(\alpha)$ 
denote the permutation of $S_n$ that results from $\alpha$ by 
replacing the $i$th smallest element of $\alpha$ by $i$ for $i=1, 
\ldots ,n$. For example, $\text{red}(3~6~8~2~9) = 2~3~4~1~5$. 
Given a permutation $\tau \in S_j$ and permutation 
$\sigma = \sigma_1 \ldots \sigma_n \in S_n$, we say that 
\begin{enumerate}
\item  $\tau$ {\bf occurs} in $\sigma$ if there are
$1 \leq i_1 < i_2 < \cdots < i_j \leq n$ such that
$\text{red}(\sigma_{i_1} \ldots \sigma_{i_j}) = \tau$,

\item $\sigma$ {\bf avoids} 
$\tau$ if there is no occurrence of $\tau$ in $\sigma$, and   

\item there is a {\bf $\tau$-match  in $\sigma$ starting at
position $i$} if $\text{red}(\sigma_i \sigma_{i+1} \ldots \sigma_{i+j-1}) = \tau$.
\end{enumerate}
We let $\tau\text{-mch}(\sigma)$ denote the number of $\tau$-matches in 
$\sigma$.  We say that $\tau$ {\em has the minimal overlapping property} or $\tau$ is {\em minimally overlapping} if the smallest $n$ such that 
there exists $\sigma \in S_n$ such that $\tau\text{-mch}(\sigma)=2$ is $2j-1$. This means 
in any two consecutive $\tau$-matches in a permutation $\sigma$ can share 
at most one element which must necessarily be at the end of the 
 first $\tau$-match and the start of the second $\tau$-match. 
It follows that if $\tau \in S_j$ is minimally overlapping, 
then the smallest $n$ such that there exist a $\sigma \in S_n$ such 
that $\tau\text{-mch}(\sigma) =k$ is $k(j-1)+1$.  We call these particular elements of $S_{k(j-1)+1}$ maximum packings for $\tau$ and we let
$\mathcal{MP}_{\tau,k(j-1)+1} = \{\sigma \in S_{k(j-1)+1}:\tau\text{-mch}(\sigma) = k\}$. 
We let $\text{mp}_{\tau,k(j-1)+1} =|\mathcal{MP}_{\tau,k(j-1)+1}|$

Minimal overlapping permutations are nice in that we have 
a nice expression for the generating function of 
$x^{\tau\text{-mch}(\sigma)}$ over permutations. That is, 
Duane and Remmel \cite{DR} proved the following theorem.

\begin{theorem}\label{thm:DR}
If $\tau \in S_j$ has the minimal overlapping property, then 
\begin{multline}
\sum_{n \geq 0} \frac{t^n}{n!} \sum_{\sigma \in S_n} x^{\tau\text{-mch}(\sigma)}  
p^{\coinv{\sigma}} q^{\inv{\sigma}}
=  \\ \frac{1}{1 -(t+ \sum_{n \geq 1} \frac{t^{n(j-1)+1}}{[n(j-1)+1]_{p,q}!} 
(x-1)^{n} \text{mp}_{\tau,n(j-1)+1}(p,q))}. 
\end{multline}
\end{theorem} 
Here if $\sigma = \sigma_1 \ldots \sigma_n \in S_n$, then $\inv{\sigma} = 
|\{(i,j): 1 \leq i < j \leq n \ \& \ \sigma_i > \sigma_j\}|$ and 
$\coinv{\sigma} = 
|\{(i,j): 1 \leq i < j \leq n \ \& \ \sigma_i < \sigma_j\}|$ and 
$$\text{mp}_{\tau,n(j-1)+1}(p,q) = \sum_{\sigma \in \mathcal{MP}_{\tau,n(j-1)+1}} 
p^{\coinv{\sigma}} q^{\inv{\sigma}}.$$

Harmse and Remmel generalized these definitions and results to 
$\mathcal{F}_{n,k}$. However, as we shall see their 
definitions make perfectly good sense in the 
the setting of $\mathcal{C}_{=k[n-1]}$ or $\mathcal{ST}(n^k)$. 
That is, if 
$F$ is any filling of a $k \times n$-rectangle with distinct positive 
integers such that elements in each column increase, 
reading from bottom to top, then 
we let $\text{red}(F)$ denote the element of 
$\mathcal{F}_{n,k}$ which results from $F$ by replacing 
the $i$th smallest element of $F$ by $i$. 
For example, Figure \ref{fig:red} pictures 
a filling $F$  with its corresponding reduced filling, $\text{red}(F)$.

%\fig{1.1}{red}{An example of $\text{red}(F)$.}
\begin{figure}[htbp]
	\begin{center}
		%% Don't use epsfile!
		\scalebox{1.2}{\includegraphics{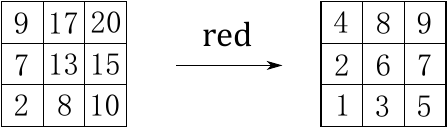}}
		\caption{An example of $\text{red}(F)$.}
		\label{fig:red}
	\end{center}
\end{figure}

If $F \in \mathcal{F}_{n,k}$ and $1 \leq c_1 < \cdots < c_j \leq n$, then 
we let $F[c_1,\ldots,c_j]$ be the filling of the $k \times j$ rectangle 
where the elements in column $a$ of $F[c_1,\ldots,c_j]$ equal the elements 
in column $c_a$ in $F$ for $a = 1, \ldots, j$.   We can 
then extend the usual pattern matching definitions from permutations 
to elements of  $\mathcal{F}_{n,k}$, $\mathcal{C}_{=k[n-1]}$, 
or $\mathcal{ST}(n^k)$ as follows. 
\begin{definition}\label{def1} Let $P$ be an element of 
$\mathcal{F}_{j,k}$ $(\mathcal{C}_{=k[j-1]}$, $\mathcal{ST}(j^k))$
and $F \in \mathcal{F}_{n,k}$ $(\mathcal{C}_{=k[n-1]}$, $\mathcal{ST}(n^k))$ 
where $j \leq n$. 
Then we say
\begin{enumerate}
\item  $P$ {\bf occurs} in $F$ if there are
$1 \leq i_1 < i_2 < \cdots < i_j \leq n$ such that
$\text{red}(F[i_1, \ldots, i_j]) = P$,

\item $F$ {\bf avoids} 
$P$ if there is no occurrence of $P$ in $F$, and   

\item there is a {\bf $P$-match  in $F$ starting at
position $i$} if $\text{red}(F[i,i+1, \ldots, i+j-1]) = P$.
\end{enumerate}
\end{definition}
We let $P\text{-mch}(F)$ denote the 
number of $P$-matches in $F$. For example, 
if we consider the fillings $P \in \mathcal{F}_{3,3}$ and 
$F,G \in \mathcal{F}_{6,3}$ shown in Figure \ref{fig:pmatch}, 
then it is easy to see that there are no $P$-matches in $F$ 
but there is an occurrence of $P$ in $F$  since 
$\text{red}(F[1,2,5]) =P$. Also, there are 2 $P$-matches in $G$ starting at 
positions 1 and 2, respectively, so $P\text{-mch}(G) =2$.

%\fig{1.1}{Pmatch}{Examples of $P$-matches and occurrences of $P$.} 
\begin{figure}[htbp]
	\begin{center}
		%% Don't use epsfile!
		\scalebox{1.15}{\includegraphics{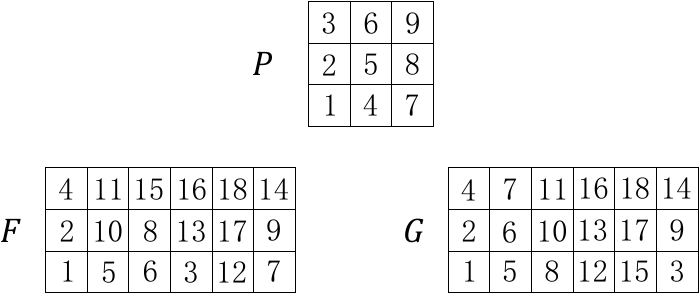}}
		\caption{Examples of $P$-matches and occurrences of $P$.}
		\label{fig:pmatch}
	\end{center}
\end{figure}

We  say that an element $P$ in  $\mathcal{F}_{j,k}$ 
($\mathcal{C}_{=k[n-1]}$, $\mathcal{ST}(n^k)$) 
has the {\em minimal overlapping property} or is {\em minimally overlapping}
if the smallest $i$ such that there exists an 
$F \in \mathcal{F}_{i,k}$ ($\mathcal{C}_{=k[i-1]}$, $\mathcal{ST}(i^k)$) 
with $P\text{-mch}(F) =2$ is $2j-1$.  Thus 
if $P$ has the minimal overlapping property, then 
two consecutive $P$-matches can only overlap in a single 
column, namely the last column of the first match and first column 
of the second match. We say that $P$ is {\em overlapping} if 
$P$ is not minimal overlapping.

If $P \in \mathcal{F}_{j,k}$ ($\mathcal{C}_{=k[j-1]}$, $\mathcal{ST}(j^k)$) 
has the minimal overlapping property, 
then the smallest $i$ such that there 
exists an $F \in \mathcal{F}_{i,k}$ 
($\mathcal{C}_{=k[i-1]}$, $\mathcal{ST}(i^k)$) such that $P\text{-mch}(F) = n$ is 
$n(j-1)+1$.  If $P \in \mathcal{F}_{j,k}$, 
we let $\mathcal{MP}_{P,n(j-1)+1}$ denote the set of all 
$F \in \mathcal{F}_{n(j-1)+1,k}$ such that $P\text{-mch}(F) =n$ 
and set $\text{mp}_{P,n(j-1)+1} = |\mathcal{MP}_{P,n(j-1)+1}|$.
If $P \in \mathcal{C}_{=k[j-1]}$, 
we let $\mathcal{CMP}_{P,n(j-1)+1}$ 
denote the set of all $F \in \mathcal{C}_{=k[n(j-1)+1]}$ 
such that $P\text{-mch}(F) =n$ and set $c\text{mp}_{P,n(j-1)+1} =
|\mathcal{CMP}_{P,n(j-1)+1}|$.
If $P \in \mathcal{ST}(j^k)$, 
we let $\mathcal{STMP}_{P,n(j-1)+1}$ denote the set of all 
$F \in \mathcal{ST}((n(j-1)+1)^k)$ such that $P\text{-mch}(F) =n$ 
and set $st\text{mp}_{P,n(j-1)+1} = |\mathcal{STMP}_{P,n(j-1)+1}|$.
In each case we shall call an 
$F \in \mathcal{MP}_{P,n(j-1)+1}$ 
($F \in \mathcal{CMP}_{P,n(j-1)+1}$, $F \in \mathcal{STMP}_{P,n(j-1)+1}$)
a {\em maximum packing} for $P$.

Given $P\in\mathcal F_{j,k}$, let
$$
A_P(x,t)=\sum_{n\geq0}\frac{t^n}{(kn)!}\sum_{F\in\mathcal{F}_{n,k}} 
x^{P\text{-mch}(F)}.
$$
Clearly, when $x=0$, 
$$
A_P(0,t)=\sum_{n\geq0}\frac{t^n}{(kn)!}\left|\{F\in\mathcal{F}_{n,k}:
P\text{-mch}(F) =0\}\right|.
$$
For $P,Q \in\mathcal F_{j,k}$, if $A_P(0,t)=A_Q(0,t)$, we say $P$ and $Q$ are c-Wilf equivalent. If $A_P(x,t)=A_Q(x,t)$, we say $P$ and $Q$ are strongly c-Wilf equivalent. In \cite{N}, in the case where 
$k=1$, Nakamura conjectured that if two permutations are c-Wilf equivalent then they are also strongly c-Wilf equivalent. Harmse and Remmel gave a similar conjecture in \cite{HR} when $k \geq 2$. That is, they made the following conjecture. 
\begin{conjecture}
$P,Q\in\mathcal F_{n,k}$ are c-Wilf equivalent if and only if $P$ and $Q$ are strongly c-Wilf equivalent.
\end{conjecture}
It has been shown that the conjecture holds for minimal overlapping patterns and the first and the last column of a pattern determines which c-Wilf equivalence class it belongs to (see \cite{DK}, \cite{DR}, \cite{E1}, and \cite{HR}). 
The key to proving such results is to prove an analogue 
to Theorem \ref{thm:DR}.  This was done by 
Harmse and Remmel \cite{HR} who proved the following theorem for 
minimal overlapping patterns in $\mathcal{F}_{j,k}$. 

\begin{theorem} \label{thm:main}
 Suppose that $k \geq 2$, $j \geq 2$, and $P \in  \mathcal{F}_{j,k}$ has the minimal overlapping property.
 Then 
 \begin{equation}\label{eq:main}
A_{P}(t,x) = \frac{1}{1 -(\frac{t}{k!}+ \sum_{n \geq 1} \frac{t^{n(j-1)+1}}{(k(n(j-1)+1))!} (x-1)^{n} \text{mp}_{P,n(j-1)+1})}. 
 \end{equation}
\end{theorem}

One can prove analogous results for 
minimal overlapping patterns $P$ in $\mathcal{C}_{=k[j-1]}$ and for 
minimal overlapping patterns $P$ in $\mathcal{ST}(j^k)$. However,  
such results can no longer be expressed just in terms of maximum 
packings, but require a more sophisticated concept of generalized 
maximum packing. This will be the subject of a forthcoming paper 
\cite{PR}.

The main focus of this paper is the following questions. 
\begin{flushleft}
\hspace{27mm}How many minimal overlapping patterns are there in $\mathcal{F}_{n,k}$? \\
\hspace{27mm}How many minimal overlapping patterns are there in $\mathcal{ST}(n^k)$? \\ 
\hspace{27mm}How many minimal overlapping patterns are there in $\mathcal{C}_{=k[n-1]}$?
\end{flushleft}

We let $M_{n,k}$ ($STM_{n,k}$, $CM_{n,k}$) denote the number of minimal overlapping patterns in $\mathcal F_{n,k}$ ( $\mathcal{ST}(n^k)$, $\mathcal{C}_{=k[n-1]}$).
Then we will be interested in the following quantities:
\begin{eqnarray*}
a_{n,k} &=&  \frac{M_{n,k}}{|\mathcal{F}_{n,k}|}, \\
b_{n,k} &=& \frac{STM_{n,k}}{|\mathcal{ST}(n^k)|}, \ \mbox{and} \\
c_{n,k} &=&  \frac{CM_{n,k}}{|\mathcal{C}_{=k[n-1]}|}.  
\end{eqnarray*}

B\'ona studied the behavior of $a_{n,1}$ in \cite{B}. In \cite{B}, 
B\'ona proved that $a_{n,1} \geq 3 -e  \approx 0.2817$.  That is, 
he proved that at least 28.17\% of the permutations in $S_n$ are minimimal 
overlapping. Moreover, he proved that $\lim\limits_{n \rightarrow \infty} a_{n,1}$ 
exists.  The main goal of this paper is to prove similar 
results for $a_{n,k}$, $b_{n,k}$, and $c_{n,k}$.

The outline of this paper is as follows. 
In Section 2, we shall give general lower bound for $a_{n,k}$ for all $k$. 
In Section 3, we shall give formulas for $M_{n,k}$ and $a_{n,k}$. In Section 
4, we shall study the asymptotic behavior of $\{a_{n,k}\}$ for fixed $k$. 
In Section 5, we shall analyze $b_{n,k}$ and in Section 6, we 
shall analyze $c_{n,k}$. In section 7, we shall briefly discuss some 
open questions about the  
problem of finding the number of minimal overlapping permutations 
that start with a given initial segment.

\section{Lower Bound for $a_{n,k}$}

B\'ona proved that the probability that a randomly selected pattern in $\mathcal F_{n,1}$ is minimal overlapping is at least $3-e\approx 0.282$, that is, $L_1=3-e$ is a lower bound for $\{a_{n,1}\}_{n\geq 1}$, independent on $n$. Based on B\'ona's idea in \cite{B},  we are able to find a basic lower bound $L_k$ for $\{a_{n,k}\}_{n\geq 1}$ for any given $k \geq 1$.

We say that $P\in\mathcal F_{n,k}$ is overlapping at position 
$i$ if $\text{red}(P[1, \ldots, i]) = \text{red}(P[n-i+1, \ldots, n])$. That is, 
the reduction of the first $i$ columns of $P$ is equal to the reduction 
of the last $i$ columns of $P$. 
Suppose $P\in\mathcal F_{n,k}$ is overlapping, then there exists integer $i$, $2\leq i\leq n-1$ such that $P$ is overlapping at position $i$. Furthermore, one can easily check that $P$ is overlapping if and only if there exists integer $i$, $2\leq i<\frac n2+1$ such that $P$ is overlapping at position $i$. Let $E_i$ be the event that $P$ is overlapping at position $i$. If $2\leq i\leq\frac n2+1$, the probability that $E_i$ happens $Pr(E_i)=\frac{(k!)^i}{(ik)!}$. That is, we can partition  
the elements of $\mathcal{F}_{n,k}$ by the set of elements that lie 
in the first $i$ columns of $P$ and the set of elements that lie 
in the last $i$ columns of $P$. Once 
the filling of the first $i$ columns of $P$ is fixed, there are 
$\frac{(ik)!}{(k!)^i}$ ways to arrange the elements in the last $i$ columns 
and only one of them is favorable. If $\frac n2<i<\frac n2+1$, $Pr(E_i)\leq\frac{(k!)^i}{(ik)!}$. Then the probability that a randomly selected pattern in $\mathcal F_{n,k}$ is overlapping is bounded by 
$$
Pr\left(\bigcup_{2\leq i< \frac n2+1}E_i\right)\leq\sum_{2\leq i< \frac n2+1}Pr(E_i)\leq\sum_{2\leq i< \frac n2+1}\frac{(k!)^i}{(ik)!}<\sum_{i\geq 2}\frac{(k!)^i}{(ik)!}.
$$
Then for a fixed $k$, we get a lower bound $L_k$ for $\{a_{n,k}\}_{n\geq 1}$.
\begin{theorem}\label{thm:Lk}
\begin{equation}\label{eq:Lk}
L_{k}=1-\sum_{i\geq 2}\frac{(k!)^i}{(ik)!}.
\end{equation}
\end{theorem}
Then by Theorem \ref{thm:Lk}, we can compute $L_1$ and $L_2$ directly.
\begin{eqnarray*}
L_1&=&1-\sum_{i\geq 2}\frac1{i!}=1-(-2+e)=3-e\approx0.282.\\
L_2&=&1-\sum_{i\geq 2}\frac{(2!)^i}{(2i)!}=3-\sum_{i\geq 0}\frac{2^i}{(2i)!}=3-\frac12\left(\sum_{i=0}\frac{(\sqrt{2})^i}{i!}+\sum_{i=0}\frac{(-\sqrt{2})^i}{i!}\right)\\
&=&3-\frac12\left(e^{\sqrt{2}}+e^{-\sqrt{2}}\right)=3-\cosh(\sqrt{2})\approx0.822.
\end{eqnarray*}
The results above indicate there are at least 28.2$\%$ of patterns in $\mathcal F_{n,1}$ that are minimal overlapping, and at least 82.2$\%$ of patterns in $\mathcal F_{n,2}$ that are minimal overlapping,
no matter what $n$ is.\par

Moreover, by observing Equation (\ref{eq:Lk}), one can easily find that the sequence $\{L_k\}$ is monotone increasing. It agrees with our intuition because a pattern is minimal overlapping as long as there exists one row whose reduction is minimal overlapping and, hence, more rows means higher chance to be minimal overlapping. Also, since $\{L_k\}$ is bounded by $1$, the sequence converges.
\begin{theorem}
\begin{equation}\label{eq:limLk}
\lim_{k\rightarrow \infty}L_{k}=1.
\end{equation}
\end{theorem}
\begin{proof}
Since
$$
L_k=1-\sum_{i\geq 2}\frac{(k!)^i}{(ik)!},
$$
to show (\ref{eq:limLk}) it suffices to show
$$
\lim_{k\rightarrow\infty}\sum_{i\geq 2}\frac{(k!)^i}{(ik)!}=0.
$$
For $i\geq 2$,
$$
\lim_{k\rightarrow\infty}\frac{(k!)^i}{(ik)!}=0\text{, }\phantom{123}\text{ }\left|\frac{(k!)^i}{(ik)!}\right|\leq \frac1{i^2}\phantom{123}\text{ }\text{ and }\phantom{123}\text{} \sum_{i\geq 2}\frac1{i^2}<\infty .
$$
Then by the dominated convergence theorem,
$$
\lim_{k\rightarrow\infty}\sum_{i\geq 2}\frac{(k!)^i}{(ik)!}=
\sum_{i\geq 2}\lim_{k\rightarrow\infty}\frac{(k!)^i}{(ik)!}=0.
$$
\end{proof}
Thus we have the following corollary. 
\begin{corollary}
For a given $n$,
\begin{equation}
\lim_{k\rightarrow \infty}a_{n,k}=1.
\end{equation}
\end{corollary}
Now we see no matter how large $n$ is, almost every pattern in $\mathcal F_{n,k}$ is minimal overlapping as long as $k$ is large enough. Another fact worth mentioning  is that $\{L_k\}$ converges to $1$ rapidly. By (\ref{eq:Lk}), $L_1 \approx 0.282$, $L_2 \approx 0.822$, $L_3 \approx 0.950$, $L_4 \approx 0.986$, $L_5 \approx 0.996$, $\cdots$.

\section{Formulas for $M_{n,k}$ and $a_{n,k}$}
It follows from our observations in the previous section  
that a pattern $P\in\mathcal F_{n,k}$ is overlapping if and only if there exists a unique integer $i$ such that $P$ is overlapping at position $i$ and the reduction of the first $i$ columns of $P$ is minimally overlapping with $2\leq i< \frac n2+1$.
The basic idea of obtaining the formula for $M_{n,k}$, the number of minimal overlapping patterns in $\mathcal F_{n,k}$, is to subtract the number of overlapping patterns from the total number.

Assuming $n$ is fixed, we separate the discussion of $M_{n,k}$ into two cases.

Case 1. $k$ is an even number.
\begin{eqnarray*}
M_{n,k}&=&\binom {nk} {k,k,\ldots,k}-M_{2,k}\binom {nk} {2k,2k}\binom {(n-4)k} {k,k,\cdots,k}-
M_{3,k}\binom {nk} {3k,3k}\binom {(n-6)k} {k,k,\cdots,k}\\
&&-M_{4,k}\binom {nk} {4k,4k}\binom {(n-8)k} {k,k,\ldots,k}-\cdots -M_{\frac n2,k}\binom {nk} {\frac 12nk,\frac 12nk}\\
&=&\frac{(nk)!}{(k!)^n}-\sum_{i=2}^{\frac n2}M_{i,k}\binom {nk} {ik,ik}\frac{((n-2i)k)!}{(k!)^{n-2i}}.
\end{eqnarray*}

The only significant difference between the even case and the odd case is the term corresponding to patterns overlapping at position $\frac{n+1}2$.

Case 2. $n$ is an odd number.
\begin{eqnarray*}
M_{n,k}&=&\binom {nk} {k,k,\ldots,k}-M_{2,k}\binom {nk} {2k,2k}\binom {(n-4)k} {k,k,\cdots,k}-
M_{3,k}\binom {nk} {3k,3k}\binom {(n-6)k} {k,k,\cdots,k}\\
&&-\cdots -M_{\frac {n-1}2,k}\binom {nk} {\frac 12(n-1)k,\frac 12(n-1)k}-\sum_{P\in\mathcal M_{\frac{n+1}2,k}}\text{mp}_{P,n}\\
&=&\frac{(nk)!}{(k!)^n}-\sum_{i=2}^{\frac {n-1}2}M_{i,k}\binom {nk} {ik,ik}\frac{((n-2i)k)!}{(k!)^{n-2i}}-\sum_{P\in\mathcal M_{\frac{n+1}2,k}}\text{mp}_{P,n}.
\end{eqnarray*}
Dividing $M_{n,k}$ by $|\mathcal F_{n,k}|$, we get formula of $a_{n,k}$ as follows
\begin{theorem}\label{thm:evenodd}
\begin{eqnarray}
\text{If $n$ is even, }&&a_{n,k}=1-\sum_{i=2}^{\frac n2}a_{i,k}\frac{(k!)^i}{(ik)!}.\label{eq:even}\\
\text{If $n$ is odd,  }&&a_{n,k}=1-\sum_{i=2}^{\frac {n-1}2}a_{i,k}\frac{(k!)^i}{(ik)!}-b_{n,k},\label{eq:odd}
\end{eqnarray}
where $b_{n,k}=\frac{\sum_{P\in\mathcal M_{\frac{n+1}2,k}}\text{mp}_{P,n}}{|\mathcal F_{n,k}|}$.
\end{theorem}
It may seem that (\ref{eq:even}) and (\ref{eq:odd}) gives us a way to 
recursively compute $a_{n,k}$, but unfortunately this is not 
the case. That is, while the even terms depend only on the previous values, 
the odd terms rely on both previous terms and 
$\sum_{P \in \mathcal{M}_{(n+1)/2,k}} \text{mp}_{P,n}$ 
which is difficult to compute. Indeed, 
we have no general way to compute $\text{mp}_{P,n}$ for a fixed 
minimal overlapping $P$.

\section{The limit of $a_{n,k}$}
In this section, we study the asymptotic behavior of $\{a_{n,k}\}_{n\geq 1}$, for any given integer $k\geq 1$. 
Throughout this section, we shall assume that 
$k$ is a  fixed number greater than or equal to 1. 
We let $\{d_n\}$ be the subsequence of $\{a_{n,k}\}_{n\geq 1}$ consisting of odd terms, ie., $d_j=a_{2j-1,k}$, $j=1,2,3,\cdots$. We let $\{e_n\}$ be the subsequence consisting of even terms.\par
From equation (\ref{eq:even}) in Theorem \ref{thm:evenodd}, we see $\{e_n\}$ is a monotone decreasing sequence and  $a_{n,k}$ has lower bound $L_k$, hence $\{e_n\}$ is convergent. To show $a_{n,k}$ has a limit, we need to show $\{d_n\}$ converges to the limit of $\{e_n\}$. This will result in the following theorem. 
\begin{theorem}\label{thm:limit}
For a fixed integer $k\geq 1$, 
$\lim\limits_{n\rightarrow \infty}a_{n,k}$ exists.
\end{theorem}
\begin{proof}
We only need to show the limit of $d_n$ equals the limit of $e_n$. Since $d_{n+1}=a_{2n+1,k}$ and $e_n=a_{2n,k}$, by (\ref{eq:even}) and (\ref{eq:odd}), we have
$e_n-d_{n+1}=b_{2n+1,k}$.
\begin{equation}\label{eq:RHS}
b_{2n+1,k}=\frac{\sum_{P\in\mathcal M_{n+1,k}}\text{mp}_{P,2n+1}}{|\mathcal F_{2n+1,k}|}\leq \frac{ M_{n+1,k}\binom{(2n+1)k}{nk}}{|\mathcal F_{2n+1,k}|}
\end{equation}
The inequality holds because that the fact that the 
reduction of the first $n+1$ columns of a pattern in $\mathcal F_{2n+1,k}$ is minimal overlapping does not guarantee the pattern is overlapping at position $n+1$. In other words,  it is a necessary condition that the reduction of the first $n + 1$ columns is minimal overlapping but that it is not sufficient. Take patterns in $\mathcal F_{5,1}$ as an example. Assume the reduction of the first three numbers is $1~3~2$ which is minimal overlapping. If the initial three numbers are $3~4~2$, we can put $5~1$ in the end to make $3~4~2~5~1$ overlapping at the middle position. However, a pattern with prefix $3~4~2$ can not be overlapping at position $3$. We use RHS for the right-handed side of (\ref{eq:RHS}).
\begin{eqnarray*}
\text{RHS}&=&\frac{ \binom{(2n+1)k}{nk}M_{n+1,k}}{|\mathcal F_{2n+1,k}|}=
\frac{\binom{(2n+1)k}{nk} |\mathcal F_{n+1,k}|a_{n+1,k}}{|\mathcal F_{2n+1,k}|}\\
&=&\frac{(2nk+k)!(nk+k)!(k!)^{2n+1}}
{(nk)!(nk+k)!(k!)^{n+1}(2nk+k)!}a_{n+1,k}\\
&=&\frac{(k!)^n}{(nk)!}a_{n+1,k}\\
&=&\frac{a_{n+1,k}}{|\mathcal F_{n,k}|}.
\end{eqnarray*}
As $n$ goes to infinity, $|\mathcal F_{n,k}|$ goes to infinity and because $a_{n+1,k}$ is bounded by $1$,
$$
\lim_{n\rightarrow \infty}e_n-d_{n+1}=\lim_{n\rightarrow\infty}
b_{2n+1,k}=0.
$$
Therefore, $\lim\limits_{n\arrowvert\infty}{d_n}=\lim\limits_{n\arrowvert\infty}{e_n}$ and then hence $\lim\limits_{n\rightarrow \infty}a_{n,k}$ exists.
\end{proof}
Note that $|\mathcal F_{n,k}|$ increases rapidly as $n$ increases so that 
$a_{n,k}$ converges rapidly. The larger $k$ becomes, the faster $a_{n,k}$ converges.\par 
Let $l_k$ be the limit of $a_{n,k}$ as $n$ goes to infinity. Using Monte Carlo methods, we computed approximate values  for $l_k$ for $k=1,2,3,4$: $l_1\approx0.364$, $l_2\approx0.823$, $l_3\approx0.949$ and $l_4\approx 0.986$. Comparing $L_k$ and $l_k$, we find they are very close. This is not a coincidence because (\ref{eq:Lk}) is almost the same as (\ref{eq:even}) and (\ref{eq:odd}) and we know $a_{n,k}$ is close to $1$ when $n$ is large.

\section{Standard tableaux of rectangular shapes}\label{Sec:SYT}

The set of all the standard Young tableaux of rectangular shape $n^k$ is denoted by $\mathcal{ST}(n^k)$ and clearly it is a subset of $\mathcal F_{n,k}$. 
It is well-known that the cardinality of $\mathcal{ST}(n^2)$ is the $n$th Catalan number $\text{Cat}(n)=\frac{1}{n+1}\binom{2n}{n}$. Then using the similar argument of Theorem \ref{thm:Lk}, we have a lower bound for $b_{n,2}$, where $b_{n,2}$ is proportion of  minimal overlapping patterns in $\mathcal{ST}(n^2)$
$$
L^S_2 = 1 -\sum_{i\geq 2}\frac 1{\text{Cat}(i)}.
$$
To find the exact value of the lower bound, we must compute the sum $\sum_{i\geq 2}\frac1{\text{Cat}(i)}$. Define a power series $f(x)$ as follows
$$
f(x)=\sum_{i\geq 1}\frac1{\text{Cat}(i)}x^i.
$$
Then
$$
xf'(x)=x\sum_{i\geq 1}\frac i{\text{Cat}(i)}x^{i-1}=\sum_{i\geq 1}\frac i{\text{Cat}(i)}x^{i}
$$
By the recursion of Catalan numbers $\text{Cat}(i+1)=\frac{4i+2}{i+2}\text{Cat}(i)$,
\begin{eqnarray*}
xf'(x)&=&\sum_{i\geq1}\frac{i+2-2}{\text{Cat}(i)}x^i=\sum_{i\geq1}\frac{i+2}{\text{Cat}(i)}x^i-2\sum_{i\geq 1}\frac1{\text{Cat}(i)}x^i=\sum_{i\geq1}\frac{4i+2}{\text{Cat}(i+1)}x^i-2f(x)\\
&=&\sum_{i\geq1}\frac{4(i+1)-2}{\text{Cat}(i+1)}x^i-2f(x)=4\sum_{i\geq1}\frac{i+1}{\text{Cat}(i+1)}x^i-2\sum_{i\geq 1}\frac1{\text{Cat}(i+1)}x^i-2f(x)\\
&=&\frac4x(xf'(x)-x)-\frac2x(f(x)-x)-2f(x)=4f'(x)-\frac{2x+2}{x}f(x)-2.
\end{eqnarray*}
So we obtain a first-order ordinary differential equation
\begin{equation}\label{eq:ode}
f'(x)+\left(\frac{2x+2}{x^2-4x}\right)f(x)=\frac2{4-x}.
\end{equation}

Solving this differential equation (\ref{eq:ode}) with the initial value condition  $f(0)=0$ for $f(x)$ yields  
$$
\frac{(x-10) (x-4) x+24 \sqrt{4x-x^2} \arcsin\left(\frac{\sqrt{x}}2\right)}{(4-x)^3}
$$
Setting $x=1$ in this formula for $f(x)$ gives our next theorem. 
\begin{theorem}
\begin{equation}
\sum_{i\geq 1}\frac1{\text{Cat}(i)}=1+\frac{4\sqrt{3}\pi}{27}.
\end{equation}
\end{theorem}
Then by the above equation, we have the lower bound $L^S_2=1-\frac{4\sqrt{3}\pi}{27}\approx 0.194$. 

Similarly, for any given $k$, lower bounds for $b_{n,k}$, that is, lower bound for the proportion of minimal overlapping patterns in $\mathcal{ST}(n^k)$ is 
$$
L_k^S=1-\sum_{i\geq 2}\frac1{Y_{i^k}},
$$
where $Y_{i^k}=|\mathcal{ST}(i^k)|$ which could be counted by the hook-length formula. We computed the following  approximations for $L_k$  for $k=3,4,5,6$.
$$
L_{3}^S\approx 0.774, \phantom{123}L_{4}^S\approx 0.926, \phantom{123}L_{5}^S\approx    0.976
, \phantom{123}L_{6}^S\approx 0.992.
$$
It is easy for one to show that $\lim\limits_{k\rightarrow \infty}L_{k}^S=1$ which implies almost every pattern in $\mathcal{ST}(n^k)$ is minimal overlapping as $k$ is very large.

Next we shall find an upper bound for the proportion of minimal overlapping standard tableaux of shape $n^2$. This is equivalent to finding a lower bound for 
the proportion of overlapping patterns. For a pattern $P\in \mathcal{ST}(n^2)$, $n\geq 3$, we only consider patterns where the reduction of the first two columns is the same as the reduction of the last two columns. There are only two cases for such overlapping patterns. For Case 1 (Figure \ref{fig:ubcase1}), the number of such overlapping patterns in $\mathcal{ST}((n+2)^2)$ is equal to $|\mathcal{ST}(n^2)|$. 
%\fig{1.1}{ubcase1}{Reducing via the patterns in the first two columns for Case 1}
\begin{figure}[htbp]
	\begin{center}
		%% Don't use epsfile!
		\scalebox{1.1}{\includegraphics{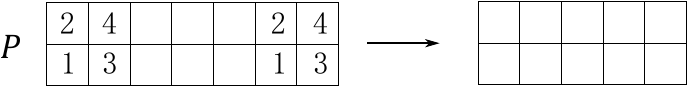}}
		\caption{Reducing via the patterns in the first two columns for Case 1}
		\label{fig:ubcase1}
	\end{center}
\end{figure}

For Case 2 (Figure \ref{fig:ubcase2}), the number of such overlapping patterns in $\mathcal{ST}((n+2)^2)$ is equal to the number of standard Young tableaux of skew shape  $(n+2,n) / (2)$.
%\fig{1.1}{ubcase2}{Reducing via the patterns in the first two columns for Case 2}
\begin{figure}[htbp]
	\begin{center}
		%% Don't use epsfile!
		\scalebox{1.1}{\includegraphics{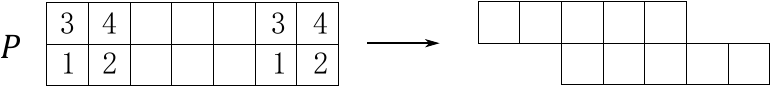}}
		\caption{Reducing via the patterns in the first two columns for Case 2}
		\label{fig:ubcase2}
	\end{center}
\end{figure}
It is easy to see that the number of standard tableaux of shape 
 $(n+2,n) / (2)$ is equal to the number of standard tableaux of 
shape $(n+2,n)$ where 1 and 2 lie in the first row. But then 
the number of standard tableaux of 
$(n+2,n)$ where 1 and 2 lie in the first row is equal to the number of 
standard tableaux of shape $(n+2,n)$ minus the number of standard tableaux 
of shape $(n+2,n)$ where 1 and 2 lie first column.  But 
then the number of standard tableaux 
of shape $(n+2,n)$ where 1 and 2 lie first column is equal to the number of standard tableaux 
of shape $(n+1,n-1)$. Using the hook length formula for the number of 
standard tableaux,  it follows that 
\begin{eqnarray*}
|\mathcal{ST}(n+2,n)/(2)| &=& |\mathcal{ST}(n+2,n)|- |\mathcal{ST}(n+1,n-1)| \\
&=& \frac{3(2n+2)!}{(n+3)!n!} - \frac{3(2n)!}{(n+2)!(n-1)!} \\
&=& \frac{9n^2+9n+6}{(n+3)(n+2)(n+1)}\binom{2n}{n}.
\end{eqnarray*}
Since the lower bound for number of overlapping patterns in $\mathcal{ST}((n+2)^2)$ is $|\mathcal{ST}(n^2)|+|\mathcal{ST}((n+2,n) / (2))|$, an upper bound for $b_{n+2,2}$, $n\geq 3$ is
\begin{eqnarray*}
U_{n+2,2}^S &=&1-\frac{|\mathcal{ST}(n^2)|+|\mathcal{ST}((n+2,n) / (2))|}{|\mathcal{ST}((n+2)^2)|}\\
&=&\frac{3 n^2+9n}{8 n^2+16n+6}.
\end{eqnarray*}

As $n$ goes to infinity, $U_{n+2,2}^S$ converges to $0.375$. Therefore, $\{b_{n,2}\}$ is asymptotically between $0.194$ and $0.375$.

\section{Generalized Euler permutations}
Elements in $\mathcal{C}_{=k[n-1]}$ are called generalized Euler permutations 
or $k$-up-down permutations. Note that 
$2$-up-down permutations are usually called up-down permutations. 
For example,  $\sigma=2~3~5~8~1~4~6~7$ is a $4$-up-down permutation of length $8$ and $\mu=1~5~2~6~3~4$ is an up-down permutation permutation of length $6$.

Next we give a lower bound $L^E_{k}$ for $\{c_{n,k}\}_{n\geq 2}$, that is, 
the proportion of minimal overlapping patterns in $\mathcal{C}_{=k[n-1]}$. Using essentially  the same argument that we used to prove 
Theorem \ref{thm:Lk}, we have a lower bound $L^E_{k}$
$$
L^E_k=1-\sum_{j\geq 2}\frac{1}{|\mathcal{C}_{=k[j-1]}|},
$$
where $|\mathcal{C}_{=k[j-1]}|$ is the number of $k$-up-down permutations of length $kj$. The generating function of $|\mathcal{C}_{=k[j-1]}|$ is
$$\sum_{n \geq 0}  |\mathcal{C}_{=k[j-1]}| \frac{t^{kn}}{(kn)!} = 
\frac{1}{\sum_{n \geq 0} \frac{(-1)^n t^{kn}}{(kn)!}}.$$ 
See Stanley's book, Chapter 3, page 389 \cite{Stan}. 
It is apparent that for fixed $k$, $L^E_k$ converges so we can calculate numerical approximation
\begin{eqnarray*}
	L^E_2&=&1-\sum_{j\geq 2}\frac{1}{|\mathcal{C}_{=2[j-1]}|}=1-\frac{1}{5}-\frac{1}{61}-
	\frac{1}{1382}-\frac1{50521}-\cdots\approx 0.783,\\
L^E_3&=&1-\sum_{j\geq 2}\frac{1}{|\mathcal{C}_{=3[j-1]}|}=1-\frac{1}{19}-\frac{1}{1513}-
\frac{1}{315523}-\cdots\approx 0.947,\\
L^E_4&=&1-\sum_{j\geq 2}\frac{1}{|\mathcal{C}_{=4[j-1]}|}=1-\frac{1}{69}-\frac{1}{33661}-
\frac{1}{60376809}-\cdots\approx 0.985,\\
L^E_5&=&1-\sum_{j\geq 2}\frac{1}{|\mathcal{C}_{=5[j-1]}|}=1-\frac{1}{251}-\frac{1}{750751}-
\frac{1}{11593285251}-\cdots\approx 0.996.
\end{eqnarray*}
It is easy for one to show that $\lim\limits_{k\rightarrow \infty}L^E_k=1$ which means almost all patterns in $\mathcal{C}_{=k[n-1]}$ are minimal overlapping when $k$ is large.

Next we will find an upper bound for $c_{n,2}$. It is well-known that numbers of up-down permutations are Euler numbers. Suppose $E(n)$ is the $n$th Euler number, then $|\mathcal{C}_{=2[n-1]}|=E(2n)$. By \cite{S}, the ratio of two adjacent Euler numbers have asymptotic estimation as follows
\begin{equation}\label{eq:asym}
\frac{E(n+1)}{E(n)} \sim \frac{2(n+1)}{\pi}.
\end{equation}

To find an upper bound of minimal overlapping patterns, we only need to find a lower bound for overlapping up-down permutations. Similar to Section \ref{Sec:SYT}, we only consider  patterns where the reduction of the first two columns is the same as the reduction of the last two columns. In other words, we only consider prefixes and suffixes of length $4$ in an up-down permutation and they have the same reduction.  In \cite{PR1}, the number of such overlapping up-down permutations of length $2n$ is given as
$$
13E(2n)-32nE(2n-1)+10n(2n-1)E(2n-2),
$$
where $n\geq 4$.
Hence we can get an upper bound for the percentage of minimal overlapping patterns in $\mathcal{C}_{=2[n-1]}$.
$$
U^E_{n,2}=1-\frac{13E(2n)-32nE(2n-1)+10n(2n-1)E(2n-2)}{E(2n)}.
$$

Applying (\ref{eq:asym}), we get an asymptotic upper bound 
$$
U^E_2=\lim\limits_{n\rightarrow \infty}U^E_{n,2}=8\pi-\frac{5}{4}\pi^2-12\approx 0.795.
$$
Therefore, $\{c_{n,2}\}$ is asymptotically between $0.783$ and $0.795$.

\section{Open questions}

A natural extension of our results would be to find the 
proportion of minimal overlapping patterns in 
$\mathcal{F}_{n,k}$ ($\mathcal{C}_{=k[n-1]}$, $\mathcal{ST}(n^k)$) 
whose first $j$ columns is equal to some fixed $P \in \mathcal{F}_{j,k}$ 
among all the elements of $\mathcal{F}_{n,k}$ ($\mathcal{C}_{=k[n-1]}$, $\mathcal{ST}(n^k)$) whose first $j$ columns is equal to $P$. For example, in 
the simplest case, we would be interested in the question of 
how many permutations in $S_n$ starting with $m$ are minimal overlapping.

It is not hard to see that for a fixed $m$, as $n$ approaches to infinity, there are at least $(3-e)(n-1)!$ permutations in $S_n$ starting with $m$ are minimal overlapping. Similar to Theorem \ref{thm:limit}, we can show that the proportion of minimal overlapping patterns in $S_n$ starting with $m$ converges as $n$ goes to infinity.

We have used Monte Carlo methods to estimate  the limit of the proportion of minimal overlapping permutations starting among all permutations 
that start with $m$ for $m=1,2,3,4,5,6,7$. Our computations yielded 
the following estimates:  
$0.392$, $0.384$, $0.375$, $0.368$, $0.365$, $0.361$ and $0.358$ respectively. Noting that the sequence is monotone decreasing, we ask whether 
this is true in general.  That is, if $1 \leq a < b \leq \lceil\frac{n}{2}\rceil$, is it the 
case that as $n$ approaches infinity, is the proportion of 
minimal overlapping permutations in $S_n$ that start with $a$ among all the 
permutations of $S_n$ that start with $a$ greater than the proportion of 
minimal overlapping permutations in $S_n$ that start with $b$ among  all the 
permutations of $S_n$ that start with $b$?

\end{document}